\documentclass[11pt,a4paper]{article}

\usepackage[T1]{fontenc}
\usepackage{lmodern}
\usepackage{microtype}
\usepackage{amsmath,amssymb,amsthm,mathtools}
\usepackage[margin=1in]{geometry}
\usepackage{enumitem}
\usepackage[hidelinks]{hyperref}

\newtheorem{theorem}{Theorem}[section]
\newtheorem{proposition}[theorem]{Proposition}
\newtheorem{lemma}[theorem]{Lemma}
\newtheorem{corollary}[theorem]{Corollary}
\theoremstyle{definition}

\theoremstyle{remark}

\newcommand{\T}{\mathbb T}
\newcommand{\R}{\mathbb R}
\newcommand{\Z}{\mathbb Z}
\newcommand{\distT}[1]{\left\lVert #1\right\rVert_{\R/\Z}}
\newcommand{\abs}[1]{\left|#1\right|}
\newcommand{\norm}[1]{\left\lVert #1\right\rVert}
\newcommand{\id}{\operatorname{id}}
\newcommand{\DC}{\operatorname{DC}}

\title{Regularity, quantitative deviation, and non-rigidity of a lacunary skew product}
\author{Yinshan Chang\thanks{Address: College of Mathematics, Sichuan University, Chengdu 610065, PR China; Email: ychang@scu.edu.cn}, Jian Wang\thanks{Address: School of Mathematical Sciences and LPMC, Nankai University, Tianjin 300071, PR China; Email: wangjian@nankai.edu.cn} and Junchang Zhou\thanks{Address: School of Mathematical Sciences and LPMC, Nankai University, Tianjin 300071, PR China; Email: 2120240070@mail.nankai.edu.cn}}
\date{}

\begin{document}

\maketitle

\begin{abstract}
Let $\alpha$ be irrational and let $q_j$ be the denominators of its
continued-fraction convergents.  We study the function
\[
 h(x)=\sum_{j\geq1}\frac{\cos(2\pi q_jx)}{q_j}
\]
and the skew product
\[
 f(x,y)=(x+\alpha,y+h(x))\pmod{\Z^2}.
\]
The function $h$ is H\"older continuous of every exponent below one. A Fourier argument shows that $h$ is not Lipschitz. The map $f$ is a toral pseudo-rotation with rotation vector $(\alpha,0)$, but it has neither bounded mean motion nor $C^0$-rigidity. Suppose $\alpha$ satisfies
the Diophantine condition $\DC(\tau)$. Then, $f$ has $(C,1-1/\tau)$-deviation when $\tau>1$; and it
has $(C_\delta,\delta)$-deviation for every $0<\delta<1$, but not for $\delta=0$ when $\tau=1$.
\end{abstract}


\section{Introduction and main result}

Write $\T=\R/\Z$ and
\[
 \distT{t}=\min_{m\in\Z}\abs{t-m}.
\]
Fix an irrational number $\alpha$, and let $p_j/q_j$ be its
continued-fraction convergents.  The object of this paper is the function
\begin{equation}\label{eq:def-h}
 h(x)=\sum_{j=1}^{\infty}\frac{\cos(2\pi q_jx)}{q_j}.
\end{equation}
The associated skew product and its natural lift are
\begin{align}
 f(x,y)&=(x+\alpha,y+h(x))\pmod{\Z^2},\label{eq:def-f}\\
 F(x,y)&=(x+\alpha,y+h(x)).\label{eq:def-F}
\end{align}
Skew products over irrational rotations are classical objects in topological dynamics. The bounded-mean-motion viewpoint for toral maps is developed in \cite{Jager}. For background on the Fourier tools used here, see \cite{Zygmund}.

For $n\geq 1$, set
\begin{equation}\label{eq:birkhoff-sum}
 S_nh(x)=\sum_{r=0}^{n-1}h(x+r\alpha).
\end{equation}
Then
\begin{equation}\label{eq:iterate}
 F^n(x,y)=\bigl(x+n\alpha,y+S_nh(x)\bigr).
\end{equation}

Let $g$ be a torus homeomorphism isotopic to the identity, and let $G$ be a lift.  We
call $g$ a \emph{pseudo-rotation} if there is a vector $\rho(G)\in\R^2$ such that
\[
 \frac{G^n(z)-z}{n}\longrightarrow\rho(G)
 \qquad\text{for every }z\in\R^2.
\]
Here, we do not require $\rho(G)$ to be totally irrational. Define
\begin{equation}\label{eq:deviation-def}
 D_n(g)=\sup_{z\in\R^2}\norm{G^n(z)-z-n\rho(G)}_2
\end{equation}
and
\begin{equation}\label{eq:tail-deviation}
 \Delta_N(g)=\sup_{n\geq N}\sup_{z\in\R^2}
 \norm{\frac{G^n(z)-z}{n}-\rho(G)}_2.
\end{equation}
The map satisfies the \emph{$(C,\delta)$-deviation condition}, where
$0\leq\delta<1$, if
\[
 \Delta_N(g)\leq CN^{\delta-1}\qquad(N\geq1).
\]
Since $\delta-1<0$, this is equivalent to
\[
 D_n(g)\leq Cn^\delta\qquad(n\geq1).
\]
The case $\delta=0$ is precisely \emph{bounded mean motion}.

We equip $\T^2$ with the quotient
Euclidean metric. Finally, $g$ is \emph{$C^0$-rigid} if there is a sequence $n_\ell\to\infty$ for which $g^{n_\ell}\to\id$ uniformly.

For $\tau\geq1$, we write $\alpha\in\DC(\tau)$ if there is
$\gamma>0$ such that for all $k\geq 1$,
\begin{equation}\label{eq:DC}
 \distT{k\alpha}\geq\gamma\abs{k}^{-\tau}.
\end{equation}

Throughout the paper, $C$ denotes a finite positive constant whose value may change from line to line. Under Diophantine conditions, it may depend on $\gamma$ and $\tau$, but never on the iterate $n$.

\begin{theorem}\label{thm:main}
Let $\alpha$ be irrational, define $h$ by \eqref{eq:def-h}, and define
$f$ by \eqref{eq:def-f}.  Then the following statements hold.
\begin{enumerate}[label=\textup{(\roman*)},leftmargin=2.2em]
 \item There is a constant $C_h$ such that, for
 $0<d=\distT{x-y}\leq1/2$,
 \begin{equation}\label{eq:log-lip-main}
  \abs{h(x)-h(y)}\leq C_h d\log\frac{e}{d}.
 \end{equation}
 Hence, $h\in C^{0,\beta}(\T)$ for every $0<\beta<1$. But $h$ is not
 Lipschitz.

 \item The map $f$ is a pseudo-rotation with rotation vector
 $\rho(F)=(\alpha,0)$. It does not have bounded mean motion.

 \item If $\alpha\in\DC(\tau)$ and $\tau>1$, then
 \begin{equation}\label{eq:power-main}
  D_n(f)\leq Cn^{1-1/\tau};
 \end{equation}
 in particular, $f$ has $(C,1-1/\tau)$-deviation.  If
 $\alpha\in\DC(1)$, then
 \begin{equation}\label{eq:log-main}
  D_n(f)\leq C\log(2n).
 \end{equation}
 Consequently, $f$ has $(C_\delta,\delta)$-deviation for every $0<\delta<1$, but it does not have $(C,0)$-deviation.

 \item The map $f$ is not $C^0$-rigid. (No Diophantine condition is needed for this conclusion.)
\end{enumerate}
\end{theorem}

The proof is self-contained apart from standard facts about
continued fractions and Fourier series.

\section{Continued-fraction estimates}

Suppose \(\alpha=[a_0; a_1, a_2, \ldots]\). We shall use
\begin{equation}\label{eq:cf-recurrence}
 q_{j+2}=a_{j+2}q_{j+1}+q_j,
 \qquad a_{j+2}\geq1,
\end{equation}
and
\begin{equation}\label{eq:cf-distance}
 \frac{1}{q_{j+1}+q_j}
 \leq\distT{q_j\alpha}
 \leq\frac{1}{q_{j+1}}.
\end{equation}
These standard continued-fraction estimates can be found, for example,
in \cite{Khinchin}.

\begin{lemma}\label{lem:lacunary}
The denominators satisfy $q_{j+2}\geq2q_j$.  Consequently, there are
constants $C$ and $C_b$, depending only on $b>0$ where indicated, such
that for every $Q\geq2$,
\begin{align}
 \#\{j:q_j\leq Q\}&\leq C\log(2Q),\label{eq:count}\\
 \sum_{q_j>Q}q_j^{-1}&\leq C Q^{-1},\label{eq:tail}\\
 \sum_{q_j\leq Q}q_j^b&\leq C_bQ^b.\label{eq:positive-sum}
\end{align}
\end{lemma}

\begin{proof}
The recurrence \eqref{eq:cf-recurrence} and the monotonicity of the
denominators give
\[
 q_{j+2}\geq q_{j+1}+q_j\geq2q_j.
\]
Thus the even and odd subsequences each grow by at least a factor of
two.  The counting estimate follows immediately.  Summing a geometric
series forward from the first denominator larger than $Q$ gives
\eqref{eq:tail}; summing backward from the largest denominator at most
$Q$ gives \eqref{eq:positive-sum}.
\end{proof}

In particular, $\sum_jq_j^{-1}<\infty$. Thus, \eqref{eq:def-h} converges uniformly and defines a continuous
one-periodic function.  Termwise integration yields
\begin{equation}\label{eq:mean-zero}
 \int_\T h(x)\,dx=0.
\end{equation}

\section{H\"older regularity and failure of Lipschitz continuity}

The next proposition proves the first part of Theorem~\ref{thm:main}.

\begin{proposition}\label{prop:regularity}
The function $h$ obeys the log-Lipschitz estimate
\eqref{eq:log-lip-main}.  It belongs to every H\"older class
$C^{0,\beta}(\T)$ with $0<\beta<1$, but it is not Lipschitz.
\end{proposition}

\begin{proof}
Let $d=\distT{x-y}$.  Choose representatives of $x$ and $y$ whose
distance is $d$; the assertion is trivial when $d=0$.  Since
\[
 \abs{\cos u-\cos v}\leq\min\{2,\abs{u-v}\},
\]
we have
\begin{equation}\label{eq:cos-bound}
 \abs{\cos(2\pi q_jx)-\cos(2\pi q_jy)}
 \leq\min\{2,2\pi q_jd\}.
\end{equation}
Split the defining series at $Q=d^{-1}$.  Lemma~\ref{lem:lacunary}
gives
\begin{align*}
 \abs{h(x)-h(y)}
 &\leq 2\pi d\sum_{q_j\leq d^{-1}}1
       +2\sum_{q_j>d^{-1}}q_j^{-1}\\
 &\leq C d\log\frac{e}{d}+Cd
 \leq C_h d\log\frac{e}{d}.
\end{align*}
For each $0<\beta<1$, the function
$d^{1-\beta}\log(e/d)$ is bounded on $0<d\leq1/2$; hence
$h\in C^{0,\beta}(\T)$.

With the Fourier convention
\[
 \widehat g(k)=\int_\T g(x)e^{-2\pi ikx}\,dx,
\]
uniform convergence of \eqref{eq:def-h} gives
\begin{equation}\label{eq:h-fourier}
 \widehat h(q_j)=\frac{1}{2q_j}.
\end{equation}
Suppose that $h$ were Lipschitz.  A Lipschitz function on the circle
is absolutely continuous and has a weak derivative
$h'\in L^\infty(\T)$.  Integration by parts in the distributional
sense gives
\[
 \widehat{h'}(k)=2\pi ik\widehat h(k).
\]
Then, \eqref{eq:h-fourier} implies
\[
 \abs{\widehat{h'}(q_j)}=\pi
 \qquad(j\geq1).
\]
This contradicts the Riemann--Lebesgue lemma, since the Fourier
coefficients of the $L^1$ function $h'$ must tend to zero.  Therefore
$h$ is not Lipschitz.
\end{proof}

\section{Rotation vector and bounded mean motion}

We first identify the rotation vector.

\begin{lemma}\label{lem:uniform-ergodic}
For every $\varphi\in C(\T)$,
\[
 \frac1n\sum_{r=0}^{n-1}\varphi(x+r\alpha)
 \longrightarrow\int_\T\varphi(t)\,dt
\]
uniformly in $x$.
\end{lemma}

\begin{proof}
For a nonconstant character $e^{2\pi ikx}$, the average equals
\[
 \frac{e^{2\pi ikx}}{n}
 \frac{1-e^{2\pi ink\alpha}}{1-e^{2\pi ik\alpha}},
\]
which tends to zero uniformly because the denominator is nonzero.
The result follows first for trigonometric polynomials and then for
continuous functions by uniform approximation.
\end{proof}

Applying the lemma to $h$ and using \eqref{eq:mean-zero}, we obtain
from \eqref{eq:iterate}
\[
 \frac{F^n(x,y)-(x,y)}{n}
 =\left(\alpha,\frac{S_nh(x)}{n}\right)
 \longrightarrow(\alpha,0)
\]
uniformly.  Hence $f$ is a pseudo-rotation and
\begin{equation}\label{eq:rotation-vector}
 \rho(F)=(\alpha,0).
\end{equation}
In particular,
\begin{equation}\label{eq:Dn-Sn}
 D_n(f)=\norm{S_nh}_\infty.
\end{equation}

\begin{proposition}\label{prop:no-bmm}
For every irrational $\alpha$, the skew product $f$ does not have
bounded mean motion.
\end{proposition}

\begin{proof}
Suppose instead that
\begin{equation}\label{eq:bounded-sums}
 \sup_{n\geq1}\norm{S_nh}_\infty<\infty.
\end{equation}
Define
\[
 g_N(x)=\frac1N\sum_{n=1}^N S_nh(x).
\]
The sequence $(g_N)$ is bounded in $L^2(\T)$, so a subsequence
converges weakly to some $g\in L^2(\T)$.  The identity
\[
 S_nh(x+\alpha)=S_{n+1}h(x)-h(x)
\]
implies
\begin{equation}\label{eq:approx-cohomology}
 g_N(x+\alpha)-g_N(x)
 =\frac{S_{N+1}h(x)-S_1h(x)}{N}-h(x).
\end{equation}
Under \eqref{eq:bounded-sums}, the right-hand side converges uniformly
to $-h(x)$.  Passing to the weak limit in \eqref{eq:approx-cohomology}
gives the $L^2$ cohomological equation
\begin{equation}\label{eq:cohomology}
 h(x)=g(x)-g(x+\alpha).
\end{equation}
Taking Fourier coefficients and using \eqref{eq:h-fourier},
\[
 \abs{\widehat g(q_j)}
 =\frac{1/(2q_j)}{\abs{1-e^{2\pi iq_j\alpha}}}.
\]
Since $\abs{1-e^{2\pi it}}\leq2\pi\distT{t}$, the convergent estimate
\eqref{eq:cf-distance} yields
\begin{equation}\label{eq:g-lower}
 \abs{\widehat g(q_j)}
 \geq\frac{q_{j+1}}{4\pi q_j}
 \geq\frac{1}{4\pi}.
\end{equation}
But $g\in L^2(\T)\subset L^1(\T)$, so the Riemann--Lebesgue lemma
requires $\widehat g(q_j)\to0$.  This contradiction proves that
\eqref{eq:bounded-sums} is impossible.  By \eqref{eq:Dn-Sn}, bounded
mean motion fails.
\end{proof}

\section{Control of deviations under Diophantine conditions}

For every $t\in\R$ and $n\geq1$, the geometric-series formula gives
\begin{equation}\label{eq:geometric}
 \abs{\sum_{r=0}^{n-1}e^{2\pi irt}}
 \leq\min\left\{n,\frac{C}{\distT{t}}\right\}.
\end{equation}
We now combine this bound with the Diophantine condition.

\begin{proposition}\label{prop:deviation}
Assume that $\alpha\in\DC(\tau)$.
\begin{enumerate}[label=\textup{(\alph*)},leftmargin=2.2em]
 \item If $\tau>1$, then
 \[
  \norm{S_nh}_\infty\leq Cn^{1-1/\tau}.
 \]
 \item If $\tau=1$, then
 \[
  \norm{S_nh}_\infty\leq C\log(2n).
 \]
\end{enumerate}
The constants may depend on $\alpha$ through the Diophantine constant
$\gamma$.
\end{proposition}

\begin{proof}
By \eqref{eq:geometric} and
\eqref{eq:DC},
\begin{align}
 \abs{S_nh(x)}
 &\leq\sum_{j\geq1}\frac1{q_j}
 \abs{\sum_{r=0}^{n-1}e^{2\pi irq_j\alpha}}\notag\\
 &\leq C_\gamma\sum_{j\geq1}q_j^{-1}
 \min\{n,q_j^\tau\}.
 \label{eq:master}
\end{align}

Suppose first that $\tau>1$, and set $Q=n^{1/\tau}$.  Splitting
\eqref{eq:master} at $Q$ and applying
Lemma~\ref{lem:lacunary} with $b=\tau-1$, we obtain
\begin{align*}
 \norm{S_nh}_\infty
 &\leq C_\gamma\left(
 \sum_{q_j\leq Q}q_j^{\tau-1}
 +n\sum_{q_j>Q}q_j^{-1}\right)\\
 &\leq C\left(Q^{\tau-1}+nQ^{-1}\right)
 \leq Cn^{1-1/\tau}.
\end{align*}

If $\tau=1$, take $Q=n$.  The low-frequency part of
\eqref{eq:master} is now a counting sum, while the high-frequency
part is controlled by \eqref{eq:tail}.  Thus
\[
 \norm{S_nh}_\infty
 \leq C\left(\#\{j:q_j\leq n\}
 +n\sum_{q_j>n}q_j^{-1}\right)
 \leq C\log(2n).
\]
\end{proof}

Equations \eqref{eq:Dn-Sn} and Proposition~\ref{prop:deviation}
prove Theorem~\ref{thm:main}(iii).  The following restatement makes
the dependence on a prescribed deviation exponent explicit.

\begin{corollary}\label{cor:delta}
Let $0<\delta<1$.  If there is $\gamma>0$ such that
\[
 \distT{k\alpha}\geq\gamma
 \abs{k}^{-1/(1-\delta)}
 \qquad(k\neq0),
\]
then $f$ has $(C,\delta)$-deviation.  If $\alpha\in\DC(1)$, the same
conclusion holds for every $0<\delta<1$.
\end{corollary}

\begin{proof}
For the first assertion, put $\tau=1/(1-\delta)>1$ in
Proposition~\ref{prop:deviation}.  For the second, combine the
logarithmic estimate with $\log(2n)\leq C_\delta n^\delta$.
\end{proof}

\section{Failure of rigidity}

\begin{proposition}\label{prop:no-rigidity}
For every irrational $\alpha$, the skew product $f$ is not
$C^0$-rigid.
\end{proposition}

\begin{proof}
Fix a sufficiently large $n$.  Choose $j=j(n)$ minimally so that
\begin{equation}\label{eq:choose-j}
 q_{j+1}\geq4n.
\end{equation}
Minimality gives $q_j<4n$.  Put
$\theta_j=\distT{q_j\alpha}$.  By \eqref{eq:cf-distance},
\[
 \theta_j\leq q_{j+1}^{-1}\leq(4n)^{-1}.
\]
The $q_j$-th Fourier coefficient of the Birkhoff sum is
\begin{equation}\label{eq:Sn-coeff}
 \widehat{S_nh}(q_j)
 =\frac{1}{2q_j}\sum_{r=0}^{n-1}e^{2\pi irq_j\alpha}.
\end{equation}
The geometric-series formula, together with
$\sin(\pi u)\geq2u$ for $0\leq u\leq1/2$ and
$\sin(\pi u)\leq\pi u$, gives
\[
 \abs{\sum_{r=0}^{n-1}e^{2\pi irq_j\alpha}}
 =\frac{\abs{\sin(\pi n\theta_j)}}
 {\abs{\sin(\pi\theta_j)}}
 \geq\frac{2n}{\pi}.
\]
It follows from $q_j<4n$ that
\begin{equation}\label{eq:uniform-lower}
 \norm{S_nh}_\infty
 \geq\abs{\widehat{S_nh}(q_j)}
 \geq\frac{n}{\pi q_j}>\frac{1}{4\pi}.
\end{equation}

Suppose, toward a contradiction, that
$f^{n_\ell}\to\id$ uniformly for some $n_\ell\to\infty$.  By
\eqref{eq:iterate}, the vertical coordinates satisfy
\begin{equation}\label{eq:mod-small}
 \varepsilon_\ell:=\max_{x\in\T}
 \distT{S_{n_\ell}h(x)}\longrightarrow0.
\end{equation}
For large $\ell$, $\varepsilon_\ell<1/4$.  Every $x$ then has a
unique integer $m_\ell(x)$ such that
\[
 \abs{S_{n_\ell}h(x)-m_\ell(x)}\leq\varepsilon_\ell.
\]
The integer-valued function $m_\ell(x)$ is continuous, hence constant
on the connected circle; denote its value by $m_\ell$.  Since
\eqref{eq:mean-zero} gives
$\int_\T S_{n_\ell}h(x)\,dx=0$, we have
\[
 \abs{m_\ell}
 \leq\int_\T\abs{m_\ell-S_{n_\ell}h(x)}\,dx
 \leq\varepsilon_\ell<\frac14.
\]
Thus $m_\ell=0$ and
$\norm{S_{n_\ell}h}_\infty\leq\varepsilon_\ell\to0$, contradicting
\eqref{eq:uniform-lower}.  Hence no rigidity sequence exists.
\end{proof}


\begin{thebibliography}{9}

\bibitem{Jager}
T.~J\"ager,
\newblock Linearization of conservative toral homeomorphisms,
\newblock \emph{Invent. Math.} \textbf{176} (2009), no.~3, 601--616,
\newblock \href{https://doi.org/10.1007/s00222-008-0171-5}
{doi:10.1007/s00222-008-0171-5}.

\bibitem{Khinchin}
A.~Ya. Khinchin,
\newblock \emph{Continued Fractions},
\newblock Dover Publications, Mineola, NY, 1997; reprint of the 1964
English translation.

\bibitem{Zygmund}
A.~Zygmund,
\newblock \emph{Trigonometric Series}, third ed.,
\newblock Cambridge Mathematical Library, Cambridge University Press,
Cambridge, 2003.

\end{thebibliography}
\end{document}